\documentclass[11pt]{article}
\usepackage{latexsym,amssymb,amsmath,ifthen,bbm,fixmath,amsthm}
\usepackage[notref,notcite,color]{showkeys}
\usepackage[nobysame,initials]{amsrefs}
\usepackage{mathtools}
\usepackage{hyperref}

\newcommand{\C}[1]{{\protect\cal #1}}

\newcommand{\e}{\varepsilon}

\newcommand{\beq}[1]{\begin{equation}\label{eq:#1}}
\newcommand{\eeq}{\end{equation}}

\newtheorem{theorem}{Theorem}
\newtheorem{lemma}[theorem]{Lemma}

\newtheorem{proposition}[theorem]{Proposition}

\newtheorem{corollary}[theorem]{Corollary}

\newtheorem{cla}{Claim}[theorem]

\newenvironment{poc}{\begin{proof}[Proof of the claim]}{\end{proof}}

\newcommand{\hide}[1]{}

\newcommand{\ptd}[1]{\ifthenelse{\equal{#1}{}}
{\cite{Pikhurko14}}
{\cite{Pikhurko14}*{#1}}}

\begin{document}
\title{\bf\Large Strong non-principality of positive codegree Tur\'an density}
\date{}
\author{Levente Bodnár\thanks{Mathematics Institute and DIMAP, University of Warwick, Coventry, UK. Research supported by ERC Advanced Grant 101020255. Email: \texttt{\{levente.bodnar,jun.gao,o.pikhurko,shumin.sun\}@warwick.ac.uk}}
\and Jun Gao\footnotemark[1]
\and Oleg Pikhurko\footnotemark[1] 
\and Mingyuan Rong\thanks{School of Mathematical Sciences, USTC, Hefei, China. Supported by National Key R\&D Program of China 2023YFA1010201, NSFC Grant No. 12125106 and the USTC Excellent PhD Students Overseas Study Program. Email:\texttt{rong\_ming\_yuan@mail.ustc.edu.cn}}
\and Shumin Sun\footnotemark[1]}

\maketitle
\begin{abstract}
    The \emph{minimum positive codegree} $\delta^+_{k-1}(G)$ of a $k$-graph $G$
    is the minimum, over all $(k-1)$-sets that lie in at least one edge, of the number of edges containing that set. The \emph{positive codegree Tur\'an density} of a $k$-graph family $\mathcal{F}$ is the asymptotically maximum value of $\delta^+_{k-1}(G)/n$ over all $\C F$-free $k$-graphs $G$ with $n\to\infty$ vertices. In this note, we establish a strong version of non-principality with respect to this density by proving that  for every $k\ge3$ there exist two $k$-graphs $F_1$ and $F_2$ such that 
    $$
    0<\gamma^+(F_1, F_2) < \min\{\gamma^+(F_1), \gamma^+(F_2)\}.
    $$
\end{abstract}
\section{Introduction}

For an integer $k \geq 2$, a \textit{$k$-uniform hypergraph} (or \textit{$k$-graph}) $H$ consists of a vertex set $V(H)$ and an edge set $E(H) \subseteq \binom{V(H)}{k}$, that is, $E(H)$ is a collection of $k$-subsets of $V(H)$. 

Given a family $\mathcal{F}$ of $k$-graphs, the \textit{Tur\'an number} $\text{ex}(n, \mathcal{F})$ is the maximum number of edges in an $n$-vertex $k$-graph that contains no member of $\mathcal{F}$ as a subgraph. 
The \textit{Tur\'an density} of $\mathcal{F}$ is defined as the limit 
\[
\pi(\mathcal{F}) := \lim_{n \to \infty} \frac{\text{ex}(n, \mathcal{F})}{\binom{n}{k}},
\]
where the existence of the limit follows from the standard averaging argument of Katona, Nemetz, and Simonovits~\cite{KNS64}.
For $k=2$, the celebrated Erd\H{o}s--Stone--Simonovits Theorem~\cite{ES46,ES66} provides a complete characterization of this density. However, for $k \geq 3$, Tur\'an problems are much more difficult. Even for simple-looking cases like the complete $3$-graph on $4$ vertices, the exact value remains unknown. For more background, we refer the reader to the surveys \cite{kee11, sido95}.

A family $\mathcal{F}$ is called \textit{non-principal} if its Tur\'an density is strictly smaller than the density of each of its members, meaning $\pi(\mathcal{F}) < \min_{F \in \mathcal{F}} \pi(\{F\})$. While no such family exists when $k=2$, Mubayi and R\"odl \cite{MR02} conjectured that non-principal families do exist for $k \geq 3$. This was first confirmed, for triple systems, by Balogh \cite{B02}. Later, Mubayi and Pikhurko \cite{MP08} extended this result by finding a non-principal family of size two for every $k \geq 3$.

Another well-studied parameter is the \textit{codegree Tur\'an density} $\gamma(\mathcal{F})$. For a $k$-graph $H$, the \emph{minimum codegree} $\delta_{k-1}(H)$ is the minimum number of edges containing any $(k-1)$-set of vertices. The \emph{codegree Tur\'an number} $\text{ex}_{k-1}(n, \mathcal{F})$ is the maximum possible value of $\delta_{k-1}(H)$ for an $n$-vertex $\mathcal{F}$-free $k$-graph, and the \emph{codegree density} is 
\[
\gamma(\mathcal{F}) := \lim_{n \to \infty} \frac{\text{ex}_{k-1}(n, \mathcal{F})}{n},
\]
where the limit is well-defined for any $k$-graph family $\mathcal{F}$, as shown in \cite[Proposition 1.2]{MubayiZhao07}. 
Mubayi and Zhao \cite{MubayiZhao07} proved the existence of finite non-principal families for $\gamma$. They also asked whether such a family of size two exists for all $k \ge 3$. For even $k \ge 4$, Sudakov (see \cite[Page 1131]{MubayiZhao07}) observed that a construction could be obtained by adapting the methods of Mubayi and Pikhurko \cite{MP08}. Recently, Gao, Pikhurko, Rong, and Sun \cite{GPRS2026rational} completely resolved this problem for all $k \ge 3$. More recently, Lin, Sun, Wang, and Zhou \cite{LSWZ26} further extended these investigations to the so-called $(k-2)$-uniform Tur\'an density, proving that non-principal families of size two exist in that setting as well.

In this note, we focus on the positive codegree version of the Tur\'an problem, which was introduced by Balogh, Lemons, and Palmer~\cite{BLP21}, and whose systematic study was initiated by 
Halfpap, Lemons, and Palmer~\cite{HLP25}. 
Formally, for a $k$-graph $H$ with $E(H)\not=\emptyset$, let $\delta^+_{k-1}(H)$ be the minimum codegree among all $(k-1)$-sets with codegree at least~$1$; if $H$ has no edges then we define $\delta^+_{k-1}(H):=0$. The \emph{positive codegree Tur\'an number} $\text{ex}^+_{k-1}(n, \mathcal{F})$ is the maximum possible value of $\delta^+_{k-1}(H)$ for an $n$-vertex $\mathcal{F}$-free $k$-graph, and the \emph{positive codegree density} is
\[
\gamma^+(\mathcal{F}) := \lim_{n \to \infty} \frac{\text{ex}^+_{k-1}(n, \mathcal{F})}{n}.
\]
The existence of the limit can be proved via a few different arguments, see~\cite{HLP25,P23On}.

One motivation behind the above definitions is to make various partite constructions (which may contain many $(k-1)$-sets not covered by a single $k$-edge but are very important in the classical Tur\'an theory) meaningful for a codegree-type problem. A number of results  on positive codegree have already been obtained; we refer the reader to \cite{BaloghHalfpapLidickyPalmer26} for an overview of the known results on this function. 
Nonetheless, many basic questions remain open.

Our main result gives an explicit non-principal family $\C F$ consisting of only two forbidden $k$-graphs. For $k$-graphs $F$ and $H$, let us abbreviate $\gamma^+(\{F\})$ and $\gamma^+(\{F,H\})$ to $\gamma^+(F)$ and $\gamma^+(F,H)$ respectively.

\begin{theorem}\label{thm:non-principal}
  For every integer $k \geq 3$, there exist $k$-graphs $F_1$ and $F_2$ such that
  \[
  0 < \gamma^+(F_1, F_2) < \min\{\gamma^+(F_1), \gamma^+(F_2)\}.
  \]
\end{theorem}

An auxiliary result that is of independent interest is a construction, for every $r\ge2$ and $k
\ge 3$, of a single forbidden $k$-graph $H_r^k$ with $\gamma^+(H_r^k)=\frac{r-1}r$, see Corollary~\ref{thm:H graph}.

\hide{
\noindent\textbf{Proof Sketch:} For any integer $k \ge 3$, our goal is to construct two graphs $F_1$ and $F_2$ such that 
$0<\gamma^+(F_1, F_2) < \min\{\gamma^+(F_1), \gamma^+(F_2)\}$. 
To construct $F_1$, we first show there exists a sequence of graphs $H_r^k$ such that 
$\gamma^+(H_r^k) = \frac{r-1}{r}$. Since $H_2^k$ is not sufficient for our main result, we define $F_1 \coloneqq Q_2^k$ using $H_2^{k-1}$. This construction ensures that $\gamma^+(F_1) = \frac{1}{2}$.
For $F_2$, we take a large clique and add all pairs to a common $(k-2)$-set, 
which ensures $\gamma^+(F_2) \ge \frac{1}{2}$. Finally, we prove that $\gamma^+(F_1, F_2) < \frac{1}{2}$ by utilizing the 
specific structures of $F_1$ and $F_2$.
}

\section{Notation and proofs}

Given a $k$-graph $H$ and a positive integer $t<k$, its \emph{$t$-shadow} $\partial_{t}(H)$ is the collection of all $t$-element vertex subsets that are contained in at least one edge of $H$.
For a vertex subset $X \subseteq V(H)$, we let $H[X]$ denote the subgraph of $H$ induced on $X$, whose edges are precisely those edges of $H$ that are fully contained in $X$.
For a subset $S \subseteq V(H)$ of size $k-1$, the neighbourhood of $S$, denoted by $N_{H}(S)$, is the set of all vertices $v \in V(H) \setminus S$ such that $S \cup \{v\} \in E(H)$.
We will omit the subscript $H$ when the underlying hypergraph is clear from the context.
For an integer $r \ge 1$, we say that a $k$-graph is \emph{$r$-colourable} if its vertices can be coloured with $r$ colours so that no edge is monochromatic.
For any $k$-graph $H$, it follows from $\delta^+_{k-1}(H) \ge \delta_{k-1}(H)$ that we have $\gamma^+(H)\ge \gamma(H)$.
We will use the following observation of Keevash and Zhao~\cite{KZ07} that follows by considering large $k$-graphs on a balanced vertex $r$-partition where we take all $k$-sets that do not lie entirely inside a part.

\begin{proposition}\label{ob:KZ}
If a $k$-graph $F$ is not $r$-colourable, then $\gamma(F)\ge \frac{r-1}{r}$. In particular, $\gamma^+(F)\ge \frac{r-1}{r}$.
\end{proposition}

Let $R$ be a $(k-1)$-uniform hypergraph and let $H$ be a $k$-uniform hypergraph. 
We define the $k$-uniform hypergraph $F(R,H)$ as follows. Let $X$ be a 
copy of $R$, called the \emph{core}, and let $S_1, \dots, S_m$ be all the $(k-1)$-edges of $X$. 
Let $H_1, \dots, H_m$ be copies of $H$ such that $X, H_1, \dots, H_m$ 
are pairwise vertex-disjoint. 
The vertex set of $F(R,H)$ is defined as
\[
V(F(R,H)) := V(X) \cup \bigcup_{j=1}^m V(H_j),
\]
and the edge set is defined as
\[
E(F(R,H)) := \bigcup_{j=1}^m \{S_j \cup \{x\} : x \in V(H_j)\} \cup \bigcup_{j=1}^m E(H_j).
\]

This operation will be useful to us because of the following result.

\begin{theorem}\label{thm: F(R,H)}
Let $r \ge 2$ and $k \ge 3$ be integers.
Let $R$ be a $(k-1)$-uniform hypergraph that is not $r$-colourable, and let $H$ be 
a $k$-uniform hypergraph that is not $(r-1)$-colourable. If $\gamma^+(R) = \frac{r-1}{r}$ 
and $\gamma^+(H) = \frac{r-2}{r-1}$, then $F = F(R,H)$ 
is not $r$-colourable and satisfies $\gamma^+(F) = \frac{r-1}{r}$. 
\end{theorem}
\begin{proof}
    First, we show that $F$ is not $r$-colourable. Suppose, for a contradiction, that there is a proper $r$-colouring of the vertex set of $F$.
    The core $X$ of $F$, which is a copy of $R$, is not $r$-colourable, so it contains a monochromatic edge $S$. 
    In $F$, the neighbourhood of $S$  cannot use this colour but spans a copy of $H$. This implies that $H$ is $(r-1)$-colourable, a contradiction. Thus $F$ is not $r$-colourable.

    By Proposition~\ref{ob:KZ}, we have $\gamma^+(F)\ge \frac{r-1}{r}$. Thus, to prove $\gamma^+(F)=\frac{r-1}{r}$, it suffices to show that $\gamma^+(F)\le \frac{r-1}{r}$.
    Take small $\varepsilon>0$ and let $n$ be sufficiently large with respect to $\varepsilon$. Let $G$ be a $k$-graph on $n$ vertices with $\delta^+_{k-1}(G)\ge \bigl(\frac{r-1}{r}+\varepsilon\bigr)n$.
    We have to show that $G$ contains a copy of $F$ as a subgraph. First we prove the following claim.
    \begin{cla}\label{claim:key}
    For any $(k-1)$-set $S$ that is contained in some edge of $G$ and any set $U$ with $|U|\le \frac{\varepsilon}{2}n$, the induced subgraph $G[N(S)\setminus U]$ contains a copy of $H$.
    \end{cla}
    \begin{poc}
    Since $S$ is contained in some edge of $G$, we have $|N(S)|\ge \bigl(\frac{r-1}{r}+\varepsilon\bigr)n$, which implies that $|N(S)\setminus U|\ge \bigl(\frac{r-1}{r}+\frac{\varepsilon}{2}\bigr)n$. Let $G'\coloneqq G[N(S)\setminus U]$.
    For any $(k-1)$-set $T\subseteq V(G')$, if $T$ is contained in some edge of $G$, then
    \[
    |N_{G'}(T)|\ge |N_G(T)|-|V(G)\setminus V(G')|\ge \left(\frac{r-1}{r}+\varepsilon\right)n-n +|V(G')|.
    \]
    It is routine to check that
    \[
    \frac{\frac {r-1}r + \e-1}{
    \frac{r - 1}{r} + \frac{\e}2}+1 - \left(\frac{r - 2}{r - 1} + \frac{\e}2\right) = \frac{\e \left(2 r^2+2
           r-2-\e r^2+\e r\right)}{2 (r-1) (2 r-2+\e r)}>0,
    \]
    which implies that $|N_{G'}(T)|\ge \bigl(\frac{r-2}{r-1}+\frac{\varepsilon}{2}\bigr)|V(G')|$.
    If $T$ is not contained in any edge of $G$, then $T$ is also not contained in any edge of $G'$. Hence either $G'$ has no edges, or its minimum positive codegree is at least
    $    \bigl(\tfrac{r-2}{r-1}+\tfrac{\varepsilon}{2}\bigr)|V(G')|$.
    
    We claim that, for $r\ge 2$, the edge set $E(G')$ is non-empty.
    Let $E$ be an edge of $G$ such that $|E\cap V(G')|$ is as large as possible.
    Suppose, for a contradiction, that $E\nsubseteq V(G')$. Let $a\coloneqq |E\cap V(G')|$. Then $E$ contains a $(k-1)$-set $T$ with $|T\cap V(G')|=a$. Since
    \[
    |N_G(T)|\ge \bigl(\tfrac{r-1}{r}+\varepsilon\bigr)n> |V(G)\setminus V(G')|,
    \]
    there exists a vertex $v\in V(G')$ such that $T\cup\{v\}\in E(G)$. This contradicts the choice of $E$.
    Thus we have $\delta_{k-1}^+(G') \ge (\frac{r-2}{r-1}+\frac{\varepsilon}{2})|V(G')|$.
    It follows from $\gamma^+(H) = \frac{r-2}{r-1}$, together with the fact that $|V(G')|$ is large enough, that $G'$ contains a copy of $H$.
    \end{poc}
It follows from $\delta_{k-1}^+(G)\ge (\frac{r-1}{r}+\varepsilon)n$ that $\delta_{k-2}^+(\partial_{k-1}(G))\ge (\frac{r-1}{r}+\varepsilon)n$.
Since $\gamma^+(R) = \frac{r-1}{r}$, there exists a copy of $R$ in $\partial_{k-1}(G)$, which we denote by $X$.
By Claim~\ref{claim:key}, the neighbourhood of every $(k-1)$-set $S\in E(X)$, even after deleting any $\frac{\e}{2}n$ vertices from it, contains a copy of $H$. Since $n$ is large enough, we can take the edges $S$ of $X$ one by one and find for each $S\in E(X)$ a copy of $H$ in $N(S)$ which is vertex-disjoint from all previous copies as well as from $X$.
These copies together with the edges to $X$ form a copy of $F$, which finishes the proof of Theorem~\ref{thm: F(R,H)}.
\end{proof}

Let us define the $k$-graph $H_r^k$ inductively on $r=1,2,\dots$ 
and $k=2,3,\dots$. For the base cases, let $H_1^k$ consist of a single $k$-edge, 
and let $H_r^2 := K_{r+1}$ for $r \ge 1$. For $r\ge 2$ and $k\ge 3$, we inductively define 
 \[
 H_r^k := F(H_r^{k-1}, H_{r-1}^k).
 \]

\begin{corollary}\label{thm:H graph}
    For any integers $r\ge 1$ and $k\ge 2$, we have that $\gamma^+(H_r^k)=\frac{r-1}{r}$ and $H_r^k$ is not $r$-colourable.
\end{corollary}
\begin{proof}
We prove the statement by induction on $r+k$. For $r=1$, the theorem holds trivially. For $k=2$, the statement follows from the celebrated Tur\'an theorem. If it holds for $(r-1,k)$ and $(r,k-1)$, then the statement for $(r,k)$ follows immediately from Theorem~\ref{thm: F(R,H)}.
\end{proof}

Let $k\ge 3$. Based on the $(k-1)$-graph $H_2^{k-1}$, we define the $k$-graph $Q_2^k$  as follows.
Let $Y$ be a copy of $H_2^{k-1}$, and let $S_1,\dots,S_m$ be all $(k-1)$-edges of~$Y$.
For each $1\le i\le m$, take new vertices $x_1^i,\dots,x_k^i$ and add the $k$-edge $\{x_1^i,\dots,x_k^i\}$ as well as the edges $S_i\cup\{x_1^i\}$ and $S_i\cup\{x_2^i\}$.
In $H_2^k$, the neighbourhood of each $S_i$ spans a $k$-edge (all $k$ of its vertices being joined to $S_i$); in $Q_2^k$, by contrast, only two vertices of this $k$-edge are joined to $S_i$.

By the definitions of $Q_2^k$ and $H_2^k$, for any $k\ge 3$, the graph $H_2^k$ contains a copy of $Q_2^k$ as a subgraph. Therefore,
\[
\gamma^+(Q_2^k)\le \gamma^+(H_2^k)\le \frac{1}{2}.
\]
Although $Q_2^k$ is 2-colourable when $k\ge 3$, it follows from the next proposition that $\gamma^+(Q_2^k)=\frac{1}{2}$.
\begin{proposition}\label{pro:F graph}
For any integer $k\ge 3$, $\gamma^+(Q_2^k)\ge \frac{1}{2}$.
\end{proposition}
\begin{proof}
    Let $G$ be a $k$-graph on $n$ vertices, and let $A$ and $B$ be a balanced partition of $V(G)$, i.e., $|A|=\lfloor\frac{n}{2}\rfloor$ and $|B|=\lceil\frac{n}{2}\rceil$.
    Define $E(G)\coloneqq\{F\in\binom{V(G)}{k}: |F\cap A|=1\}$. Clearly, $\delta^+_{k-1}(G)=\lceil\frac{n}{2}\rceil-(k-2)$.
    To prove $\gamma^+(Q_2^k)\ge \frac{1}{2}$, it is enough to show that $G$ is $Q_2^k$-free.
    Suppose, for a contradiction, that $G$ contains a copy of $Q_2^k$. Then $\partial_{k-1}(G)$ contains a copy $X$ of $H_2^{k-1}$.
    By Corollary~\ref{thm:H graph}, the $(k-1)$-graph $H_2^{k-1}$ is not $2$-colourable, so one of its edges $S$ lies entirely inside one part. Since $A$ contains no pair from the 2-shadow of $G$, we must have $S\subseteq B$.
    For such $S$, we have $N_G(S)=A$. Hence the two special neighbours of $S$ required by the definition of $Q_2^k$ both lie in $A$.
    But every edge of $G$ meets $A$ in exactly one vertex, so these two vertices cannot lie together in an edge of $G$, a contradiction.
\end{proof}

Next, we define the $k$-graph $D_r^k$ for $k\ge 3$ and $r\ge 2$.
Let $V(D_r^k)$ be the union of two disjoint sets $Z$ and $K$ with $|Z|=k-2$ and $|K|=r$.
Define
\[
E(D_r^k):=\left\{Z\cup W : W\in \binom{K}{2}\right\}.
\]
Then we have the following property.
\begin{proposition}\label{pro:D graph}
$\gamma^+(D_r^k)\ge \frac{r-2}{k+r-3}$.
\end{proposition}
\begin{proof}
    Let $G$ be a $k$-graph on $n$ vertices with a balanced partition $V(G)=V_1\cup V_2\cup\cdots\cup V_{k+r-3}$, where
    $|V_i|\in\left\{\left\lfloor\frac{n}{k+r-3}\right\rfloor,\left\lceil\frac{n}{k+r-3}\right\rceil\right\}$ for each $i\in[k+r-3]$, and let the edge set consist of all $k$-sets that intersect each part in at most one vertex. Clearly, $\delta_{k-1}^+(G) \ge \left\lfloor\frac{n}{k+r-3}\right\rfloor \cdot (r-2)$.
    Consider the $2$-shadow of $G$. It is easy to see that $\partial_2(G)$ is $K_{k+r-2}$-free. Since $\partial_2(D_r^k)=K_{k+r-2}$, it follows that $G$ is $D_r^k$-free, which implies that $\gamma^+(D_r^k)\ge\frac{r-2}{k+r-3}$.
\end{proof}

By Proposition~\ref{pro:F graph}, we have $\gamma^+(Q_2^k)\ge \frac{1}{2}$.
Moreover, by Proposition~\ref{pro:D graph}, we have $\gamma^+(D_{k+1}^k)\ge \frac{1}{2}$.
Thus, in order to prove Theorem~\ref{thm:non-principal}, it suffices to show the following lemma.
\begin{lemma}\label{lem:main}
    For any integer $k\ge 3$, we have $\frac{1}{k}\le\gamma^+(Q_2^k,D_{k+1}^k) \le \frac{1}{2}-\alpha$, where $\alpha := \frac{1}{10(k-1)}$. 
\end{lemma}
\begin{proof}
    By considering the balanced complete $k$-partite $k$-graph, it is easy to see that $ \gamma^+(Q_2^k,D_{k+1}^k) \ge \frac{1}{k}$.
    Next we will show $\gamma^+(Q_2^k,D_{k+1}^k) \le \frac{1}{2}-\alpha$.
    Let $G$ be a $k$-graph on $n$ vertices with $\delta_{k-1}^+(G)\ge\left(\frac{1}{2}-\alpha \right)n$, where $n$ is sufficiently large.
    Suppose that $G$ is $Q_2^k$-free and $D_{k+1}^k$-free.
    First, we show that $\partial_{k-1}(G)$ has large minimum positive codegree.

    \begin{cla}\label{cl:k-2degree}
    $\delta_{k-2}^+(\partial_{k-1}(G))\ge\left(\frac{1}{2}+3\alpha\right)n$.
    \end{cla}
    \begin{poc}
    For any $(k-2)$-set $S\subseteq V(G)$, let $R$ be the link graph of $S$ in $G$, i.e., 
    \[
    \mbox{$V(R)=\{v:\{v\}\cup S\in\partial_{k-1}(G)\}$ \ \ and \ \ $E(R)=\{uv:\{u,v\}\cup S\in G\}$.}
    \] Our aim is to show that $|V(R)|\ge\left(\frac{1}{2}+3\alpha\right)n$ whenever $E(R)$ is non-empty. Suppose that $E(R)$ is non-empty. Then the minimum degree of $R$ is at least $\left(\frac{1}{2}-\alpha\right)n$. Since $G$ is $D_{k+1}^k$-free, we have that $R$ is $K_{k+1}$-free. It follows from Tur\'an's theorem that $|V(R)|\ge \frac{k}{k-1}\delta_1(R)-1\ge \left(\frac{1}{2}+3\alpha\right)n$.
    \end{poc}
    
    \begin{cla}\label{cl:IndepY}
    There exists a vertex set $Y\subseteq V(G)$ of size at least $\left(\frac{1}{2}-2\alpha\right)n$ such that each edge of $G$ intersects $Y$ in at most one vertex.
    \end{cla}
    \begin{poc}
    Suppose, for a contradiction, that every set of size at least $\left(\frac{1}{2}-2\alpha\right)n$ contains at least two vertices of some edge.  By Claim~\ref{cl:k-2degree}, the $(k-1)$-graph $\partial_{k-1}(G)$ has positive codegree at least $\left(\frac{1}{2}+3\alpha\right)n$. By Corollary~\ref{thm:H graph}, it contains $H_2^{k-1}$ as a subgraph, which we denote by $X$. Let $S_1,S_2,\dots, S_m$ be all $(k-1)$-edges of $X$.
    
    We now prove that there exist disjoint edges $F_1,\dots,F_m\subseteq V(G)\setminus V(X)$ with $F_i=\{x_1^i,\dots,x_k^i\}$ such that $S_i\cup\{x_1^i\}$ and $S_i\cup\{x_2^i\}$ are edges of $G$ for every $1\le i\le m$.
    
    Suppose that we have found $F_1,\dots,F_{i-1}$ for some $i\in[m]$. Let $A:=X\cup\bigcup_{j=1}^{i-1}F_j$. Since $\delta_{k-1}^+(G)\ge\left(\frac{1}{2}-\alpha\right)n$ and $n$ is sufficiently large, we have
    \[
    \left|N(S_i)\setminus A\right|\ge\left(\frac{1}{2}-2\alpha\right)n.
    \]
    By our assumption for contradiction, there exists an edge $F_i$ intersecting $N(S_i)\setminus A$ in at least two vertices, call them $x_1^i$ and $x_2^i$. Among such edges pick one with $|F_i\setminus A|$ as large as possible. Since $\delta_{k-1}^+(G)\ge\left(\frac{1}{2}-\alpha\right)n>|A|$, a familiar argument shows that $F_i$ must be disjoint from $A$, as desired.
    Then $F_1,\dots,F_m$ together with $X$ form a copy of $Q_2^k$, a contradiction.
   \end{poc}
   Let $Y$ be a vertex set of size at least $\left(\frac{1}{2}-2\alpha\right)n$ such that each edge of $G$ intersects $Y$ in at most one vertex.
   If no edge intersects $Y$, then for any edge $E\in E(G)$ and any $(k-2)$-subset $S\subseteq E$, we have $N_{\partial_{k-1}(G)}(S)\subseteq V(G)\setminus Y$ and
   \[
   |N_{\partial_{k-1}(G)}(S)|\le n-|Y|\le\left(\frac{1}{2}+2\alpha\right)n,
   \]
   contradicting Claim~\ref{cl:k-2degree}.
   If some edge intersects $Y$ in exactly one vertex, take such an edge $E$ and choose a $(k-2)$-subset $S\subseteq E$ with $S\cap Y\neq\emptyset$. Again $N_{\partial_{k-1}(G)}(S)\subseteq V(G)\setminus Y$, so
   \[
   |N_{\partial_{k-1}(G)}(S)|\le n-|Y|\le\left(\frac{1}{2}+2\alpha\right)n,
   \]
   contradicting Claim~\ref{cl:k-2degree}. 
   Thus $G$ contains $Q_2^k$ or $D_{k+1}^k$ as a subgraph, which finishes the proof of Lemma~\ref{lem:main}.
\end{proof}

We can now formally deduce our main result. 

\begin{proof}[Proof of Theorem~\ref{thm:non-principal}.] Take $F_1:=Q_2^k$ and $F_2:=D_{k+1}^k$. By Proposition~\ref{pro:F graph} and Proposition~\ref{pro:D graph} (the latter applied with $r=k+1$), both $\gamma^+(F_1)$ and $\gamma^+(F_2)$ are at least $\frac{1}{2}$, while Lemma~\ref{lem:main} gives $0<\frac{1}{k}\le\gamma^+(F_1,F_2)\le\frac{1}{2}-\alpha<\frac{1}{2}$. Hence
\[
0<\gamma^+(F_1,F_2)<\min\{\gamma^+(F_1),\gamma^+(F_2)\},
\]
which proves Theorem~\ref{thm:non-principal}.\end{proof}

\section*{Acknowledgements}
AI (Claude Opus 4.8) was used to proofread the final draft of this paper.
\bibliography{bibexport}
\end{document}